\documentclass[a4paper,12pt]{article}

\usepackage {amsfonts, amssymb, amscd,amsthm, amsmath, eucal}

\usepackage{amsbsy}
\usepackage[all]{xy}

\usepackage{color}

\theoremstyle{plain} {
  \newtheorem{thm}{Theorem}[section]
  \newtheorem{defn}[thm]{Definition}
  \newtheorem{cor}[thm]{Corollary}
  \newtheorem{lem}[thm]{Lemma}
  \newtheorem{prop}[thm]{Proposition}
  \theoremstyle{definition}
  \newtheorem{rem}[thm]{Remark}
    \newtheorem{constr}[thm]{Construction}
  \theoremstyle{plain}
  \newtheorem{clm}[thm]{Claim}

  \newtheorem{sssection}[thm]{}
}

{

}
\renewcommand{\subsubsection}{\sssection\rm}

\newcommand{\bG}{\mathbf G}

\renewcommand{\P}{\mathbb P}

\DeclareMathOperator{\spec}{Spec}

\newcommand{\can}{\text{\rm can}}

\newcommand{\id}{\text{\rm id}}
\newcommand{\pr}{\text{\rm pr}}

\newcommand{\inc}{\text{\rm inc}}

\newcommand{\const}{\text{\rm const}}
\newcommand{\Spec}{\text{\rm Spec}}

\newcommand{\Aff}{\mathbf {A}}
\newcommand{\Pro}{\mathbf {P}}

\newcommand \xra {\xrightarrow }

\newcommand \hra {\hookrightarrow }

\newcommand{\ttf}{{\text{f}}}

\renewcommand{\P}{\mathbb P}

\newcommand\mydim{\text{\rm dim}}

\renewcommand \id{\operatorname{id}}

\renewcommand \phi\varphi

\newcommand{\et}{\text{\rm\'et}}

\newcommand{\ZZ}{\mathbb Z}

\begin{document}

\title{Nice triples and Grothendieck--Serre's conjecture concerning principal G-bundles over reductive group schemes
}

\author{Ivan Panin\footnote{The author acknowledges support of the
RNF-grant 14-11-00456.}
}


\maketitle

\begin{abstract}
In a series of papers \cite{Pan0}, \cite{Pan1}, \cite{Pan2}, \cite{Pan3} we give a detailed and better structured
proof of the Grothendieck--Serre's conjecture
for semi-local regular rings containing a finite field. The outline of the proof
is the same as in \cite{P1},\cite{P2},\cite{P3}.
If the semi-local regular ring contains an infinite field,
then the conjecture is proved in \cite{FP}. {\it Thus the conjecture
is true for regular local rings containing a field.
}

The present paper is the one \cite{Pan1} in that new series.
Theorem \ref{MainHomotopyIntrod} is one of the main result of the paper.
It is also one of the key steps
in the proof of the Grothendieck--Serre's conjecture
for semi-local rings containing a field (see \cite{Pan3}).
The proof of the main theorem is completely geometric.
It is based on an extension of theory of nice triples from \cite{PSV}
and \cite{P}.
In turn the theory of nice triples is inspired by
the Voevodsky theory of standart triples \cite{V}.
Our refinement of that Voevodsky theory
is based on the use of Artin's elementary fibrations,
on geometric lemma
\cite[Lemma 8.2]{OP2} and construction \cite[Constr. 4.2]{P}.
The latter construction is taken from \cite[the proof of lemma 8.1]{OP2}).
\end{abstract}

\section{Main results}\label{Introduction}
Let $R$ be a commutative unital ring. Recall that an $R$-group scheme $\bG$ is called \emph{reductive},
if it is affine and smooth as an $R$-scheme and if, moreover,
for each algebraically closed field $\Omega$ and for each ring homomorphism $R\to\Omega$ the scalar extension $\bG_\Omega$ is
a connected reductive algebraic group over $\Omega$. This definition of a reductive $R$-group scheme
coincides with~\cite[Exp.~XIX, Definition~2.7]{SGA3}.
A~well-known conjecture due to J.-P.~Serre and A.~Grothendieck
(see~\cite[Remarque, p.31]{Se}, \cite[Remarque 3, p.26-27]{Gr1}, and~\cite[Remarque~1.11.a]{Gr2})
asserts that given a regular local ring $R$ and its field of fractions~$K$ and given a reductive group scheme $\bG$ over $R$, the map
\[
  H^1_{\text{\'et}}(R,\bG)\to H^1_{\text{\'et}}(K,\bG),
\]
induced by the inclusion of $R$ into $K$, {\it has a trivial kernel.}
If $R$ contains an infinite field, then the conjecture is proved in [FP].

For a scheme $U$ we denote by $\mathbb A^1_U$ the affine line over $U$ and by $\P^1_U$ the projective line over $U$.
Let $T$ be a $U$-scheme. By a principal $\bG$-bundle over $T$ we understand a principal $\bG\times_UT$-bundle.
We refer to
\cite[Exp.~XXIV, Sect.~5.3]{SGA3}
for the definitions of
a simple simply-connected group scheme over a scheme
and a semi-simple simply-connected group scheme over a scheme.



\begin{thm}
\label{MainHomotopyIntrod}
Let $k$ be a field. Let $\mathcal
O$ be the semi-local ring of finitely many {\bf closed points} on a
$k$-smooth irreducible affine $k$-variety $X$.
Set $U=\spec \mathcal O$.
Let $\bG$ be a reductive
group scheme over $U$.
Let $\mathcal G$ be a principal $\bG$-bundle over $U$
trivial over the generic point of $U$.
Then there
exists a principal $\bG$-bundle $\mathcal G_t$ over the affine line $\mathbb A^1_U=\spec {\mathcal O}[t]$
and a monic polynomial $h(t) \in \mathcal O[t]$ such
that
\par
(i) the $\bG$-bundle $\mathcal G_t$ is trivial over the open subscheme $(\mathbb A^1_U)_h$ in $\mathbb A^1_U$ given by $h(t)\ne0$;
\par
(ii) the restriction of $\mathcal G_t$ to $\{0\}\times U$ coincides
with the original $\bG$-bundle $\mathcal G$.
\par
(iii) $h(1) \in \mathcal O$ is a unit.
\end{thm}
If the field $k$ is infinite a weaker result is proved in
\cite[Thm.1.2]{PSV}.
The proof of Theorem
\ref{MainHomotopyIntrod}
is given in Section
\ref{Section_Proof_Of_Theorem_1_5}.
It is easily derived
from Theorem \ref{Major2} and even from the following result,
which is weaker than  Theorem \ref{Major2}.

\begin{thm}[Geometric]\label{MajorIntrod}
Let $X$ be an affine $k$-smooth irreducible $k$-variety, and let $x_1,x_2,\dots,x_n$ be closed points in $X$.
Let $U=Spec(\mathcal O_{X,\{x_1,x_2,\dots,x_n\}})$ and $\textrm{f}\in k[X]$ be
a non-zero function vanishing at each point $x_i$.
Let $\bG$ be a reductive group scheme over $X$, $\bG_U$ be its restriction to $U$.
Then
there is a monic polinomial $h\in O_{X,\{x_1,x_2,\dots,x_n\}}[t]$,
a commutative diagram
of schemes with the irreducible affine $U$-smooth $Y$
\begin{equation}
\label{SquareDiagram2_2}
    \xymatrix{
       (\Aff^1 \times U)_{h}  \ar[d]_{inc} && Y_h:=Y_{\tau^*(h)} \ar[ll]_{\tau_{h}}  \ar[d]^{inc} \ar[rr]^{(p_X)|_{Y_h}} && X_f  \ar[d]_{inc}   &\\
     (\Aff^1 \times U)  && Y  \ar[ll]_{\tau} \ar[rr]^{p_X} && X                                     &\\
    }
\end{equation}
and a morphism $\delta: U \to Y$ subjecting to the following conditions:
\begin{itemize}
\item[\rm{(i)}]
the left hand side square
is an elementary {\bf distinguished} square in the category of affine $U$-smooth schemes in the sense of
\cite[Defn.3.1.3]{MV};
\item[\rm{(ii)}]
$p_X\circ \delta=can: U \to X$, where $can$ is the canonical morphism;
\item[\rm{(iii)}]
$\tau\circ \delta=i_0: U\to \Aff^1 \times U$ is the zero section
of the projection $pr_U: \Aff^1 \times U \to U$;
\item[\rm{(iv)}] $h(1)\in \mathcal O[t]$ is a unit;
\item[\rm{(v)}] for $p_U:=pr_U\circ \tau$ there is a $Y$-group scheme isomorphism
$\Phi: p_U^*(\bG_U) \to p_X^*(\bG)$ with $\delta^*(\Phi)=id_{\bG_U}$.
\end{itemize}
\end{thm}

{\it A sketch of the proof of Theorem \ref{MainHomotopyIntrod}}.
In general, $\bG$ does not come from $X$. However we may assume, that
as $\bG$, so $\mathcal G$ are defined over $X$ and $\bG$ is reductive over $X$.
Say, let $\mathcal G'$ is a principal $\bG$ on $X$ with $\mathcal G'|_U=\mathcal G$.
In this case there two reductive group schemes on $Y$.
Namelly, $p^*_U(\bG|_U)$ and $p^*_X(\bG)$.
Clearly, they coincides when restricted to $\delta(U)$.
By the item (v) of Theorem \ref{MajorIntrod} the scheme $Y$ can be chosen such that
two reductive group schemes
$p^*_U(\bG|_U)$ and $p^*_X(\bG)$
on $Y$ are isomorphic via
an isomorphism $\Phi$  and $\Phi$ is such that its restriction to $\delta(U)$
is the identity. Take again $p^*_X(\mathcal G')$ and regard it as
a principal $p^*_U(\bG|_U)$-bundle using the isomorphism $\Phi$.
Denote that principal $p^*_U(\bG|_U)$-bundle $_{_{U}}p^{*}_X(\mathcal G)$.
It is trivial on $Y_h$, since $p^*_X(\mathcal G')$ is trivial on $Y_h$.
Take the trivial $pr^*_U(\bG|_U)$-bundle on $\Aff^1\times U$
and glue it with $_{_{U}}p^{*}_X(\mathcal G)$ via an isomorphism over $Y_h$.
This way we get a principal $\bG|_U$-bundle $\mathcal G_t$ over $\Aff^1\times U$.
Clearly, it is the desired one.
The polinomial $h$ one should take as in Theorem \ref{MajorIntrod}.
Details are given in Section
\ref{Section_Proof_Of_Theorem_1_5}.

The article is organized as follows.
In Section \ref{NiceTriples}
we recall definition of nice triples from \cite{PSV}
and inspired by Voevodsky notion of perfect triples.

In Section \ref{SecEquatingGroups} Theorem \ref{PropEquatingGroups2} on equating group schemes
is proved. In Section \ref{NiceTriplesAndGr_Schemes} Theorem \ref{equating3_triples} is proved.
In Section \ref{First Application} a first application of the above machinery is given.
Namely, Theorems \ref{equating3} and \ref{Major2} are proved.
Theorem \ref{equating3} is a one more geometric presentation theorem.
Theorem \ref{Major2} is a stronger version of Theorem \ref{MajorIntrod}.
{\it Finally, in Section \ref{Section_Proof_Of_Theorem_1_5}
Theorem \ref{MainHomotopyIntrod}
is proved.
}
In Section \ref{PSVfinite_f} the main result of \cite{PSV}
{\it is extended to the case of an arbitrary base field}
$k$.
Namelly, the following Theorem is proved there
 \begin{thm}\label{PSV}
Let $k$ be a field. Let $\mathcal O$ be the semi-local ring of finitely many closed points on a $k$-smooth irreducible affine $k$-variety $X$
and let $K$ be its field of fractions. Let $\bG$ be an isotropic simple simply connected group scheme over $\mathcal O$.
Then for any Noetherian $k$-algebra $A$ the map
$$\text{H}^1_{et}(\mathcal O \otimes_k  A,G) \to \text{H}^1_{et}(K\otimes_k  A,G),$$
induced by the inclusion $\mathcal O$ into $K$, has trivial kernel.
\end{thm}
Stress the following. If the field $k$ is finite,
and $A=k$ and the group scheme $G$ comes from the field $k$
and it is simple simply connected,
then the latter theorem is an unpublished result due to Gabber
(see also \cite{Pan0}).


The author thanks A.~Suslin for his interest in the topic of the present article. He
also thanks to A.Stavrova for paying his attention to Poonen's works on Bertini type theorems
for varieties over finite fields. He thanks D.Orlov for useful comments concerning
the weighted projective spaces tacitely involved in the construction of elementary fibrations.
He thanks M.Ojanguren for many inspiring ideas arising from our joint works with him.

\section{Nice triples}
\label{NiceTriples}

In the present section we recall and study certain collections
of geometric data and their morphisms. The concept of a {\it nice
triple\/} was introduced in
\cite[Defn. 3.1]{PSV}
and is very similar to that of a {\it standard triple\/}
introduced by Voevodsky \cite[Defn.~4.1]{V}, and was in fact
inspired by the latter notion. Let $k$ be a field, let $X$
be a $k$-smooth
irreducible {\bf affine} $k$-variety, and let
$x_1,x_2,\dots,x_n\in X$ be {\bf a family of closed points}. Further, let
$\mathcal O=\mathcal O_{X,\{x_1,x_2,\dots,x_n\}}$ be the
corresponding geometric semi-local ring.

After substituting $k$ by its algebraic closure $\tilde k$ in $k[X]$,
we can assume that $X$ is a $\tilde k$-smooth {\it geometrically irreducible} affine $\tilde k$-scheme.
{\bf The geometric irreducibility of $X$ is required in the proposition \ref{ArtinsNeighbor}
to construct an open neighborhood $X^0$ of the family
$\{x_1,x_2,\dots,x_n\}$ and {\bf an elementary fibration} $p:X^0\to S$,
where $S$ is an open sub-scheme of the projective space
$\Pro^{\mydim X-1}_k$. The proposition \ref{ArtinsNeighbor}
is used in turn to prove the proposition \ref{BasicTripleProp_1}.
To simplify the notation, we will
continue to denote this new $\tilde k$ by $k$.
\begin{defn}
\label{DefnNiceTriple} Let $U:=\text{\Spec}(\mathcal
O_{X,\{x_1,x_2,\dots,x_n\}})$.
A \emph{nice triple} over $U$ consists of the following data:
\begin{itemize}
\item[\rm{(i)}] a smooth morphism $q_U:\mathcal X\to U$, where $\mathcal X$ is an irreducible scheme,
\item[\rm{(ii)}] an element $f\in\Gamma(\mathcal X,\mathcal
O_{\mathcal X})$,
\item[\rm{(iii)}] a section $\Delta$ of the morphism $q_U$,
\end{itemize}
subject to the following conditions:
\begin{itemize}
\item[\rm{(a)}]
each irreducible component of each fibre of the morphism $q_U$
has dimension one,
\item[\rm{(b)}]
the module
$\Gamma(\mathcal X,\mathcal O_{\mathcal X})/f\cdot\Gamma(\mathcal X,\mathcal O_{\mathcal X})$
is finite as
a $\Gamma(U,\mathcal O_{U})=\mathcal O$-module,
\item[\rm{(c)}]
there exists a finite surjective $U$-morphism
$\Pi:\mathcal X\to\Aff^1\times U$;
\item[\rm{(d)}]
$\Delta^*(f)\neq 0\in\Gamma(U,\mathcal O_{U})$.
\end{itemize}
There are many choices of the morphism $\Pi$. Any of them is regarded as assigned to
the nice triple.
\end{defn}

\begin{rem}\label{rem: NiceTriples}
Since $\Pi$ is a finite morphism, the scheme $\mathcal X$ is affine.
We will write often below $k[\mathcal X]$ for $\Gamma(\mathcal X,\mathcal O_{\mathcal X})$.
The only requirement on the morphism $\Delta$ is this: $\Delta$ is a section of $q_U$.
Hence $\Delta$ is a closed embedding. We write $\Delta(U)$ for the image of this closed embedding.
The composite map $\Delta^* \circ q^*_U: k[\mathcal X]\to \mathcal O$ is the identity.
If $Ker=Ker(\Delta^*)$, then $Ker$ is the ideal defining the closed subscheme $\Delta(U)$ in $\mathcal X$.
\end{rem}

\begin{defn}
\label{DefnMorphismNiceTriple}
A \emph{morphism}
between two nice
triples
over $U$
$$(q^{\prime}: \mathcal X^{\prime} \to U,f^{\prime},\Delta^{\prime})\to(q: \mathcal X \to U,f,\Delta)$$
is an \'{e}tale morphism of $U$-schemes $\theta:\mathcal
X^{\prime}\to\mathcal X$ such that
\begin{itemize}
\item[\rm{(1)}] $q^{\prime}_U=q_U\circ\theta$, \item[\rm{(2)}]
$f^{\prime}=\theta^{*}(f)\cdot h^{\prime}$ for an element
$h^{\prime}\in\Gamma(\mathcal X^{\prime},\mathcal O_{\mathcal
X^{\prime}})$,
\item[\rm{(3)}] $\Delta=\theta\circ\Delta^{\prime}$.
\end{itemize}
\end{defn}
Two observations are in order here.
\par\smallskip
$\bullet$ Item (2) implies in particular that $\Gamma(\mathcal
X^{\prime},\mathcal O_{\mathcal
X^{\prime}})/\theta^*(f)\cdot\Gamma(\mathcal X^{\prime},\mathcal
O_{\mathcal X^{\prime}})$ is a finite
$\mathcal O$-module.
\par\smallskip
$\bullet$ It should be emphasized that no conditions are imposed
on the interrelation of $\Pi^{\prime}$ and $\Pi$.
\par\smallskip

}

Let $U$ be as in Definition \ref{DefnNiceTriple} and
$\can:U\hra X$ be the canonical inclusion of schemes.
\begin{defn}
\label{SpecialNiceTriples}
A nice triple
$(q_U: \mathcal X \to U, \Delta, f)$ over $U$
is called special if
the set of closed points of $\Delta(U)$ is
contained in the set of closed points of $\{f=0\}$.
\end{defn}

\begin{rem}
\label{NiceToSpecialNice}
Clearly the following holds:
let $(\mathcal X,f,\Delta)$ be a special nice triple
over $U$ and let
$\theta: (\mathcal X^{\prime},f^{\prime},\Delta^{\prime}) \to (\mathcal X,f,\Delta)$
be a morphism between nice triples over $U$. Then the triple
$(\mathcal X^{\prime},f^{\prime},\Delta^{\prime})$
is a special nice triple over $U$.
\end{rem}

\begin{prop}
\label{BasicTripleProp_1}
One can shrink $X$ such that $x_1,x_2, \dots , x_n$ are still in $X$ and $X$ is affine, and then construct a special nice triple
$(q_U: \mathcal X \to U, \Delta, f)$ over $U$ and an essentially smooth morphism $q_X: \mathcal X \to X$ such that
$q_X \circ \Delta= can$, $f=q^*_X(\text{f} \ )$.
\end{prop}

\begin{proof}[Proof of Proposition \ref{BasicTripleProp_1}]
If the field $k$ is infinite, then this proposition is proved in
\cite[Prop. 6.1]{PSV}.
So, we may and will assume that $k$ is finite.
To prove the proposition repeat literally the proof of
\cite[Prop. 6.1]{PSV}. One has to replace the references to
\cite[Prop. 2.3]{PSV} and
\cite[Prop.2.4]{PSV}
with the reference to
\cite[Prop. 2.3]{P}
and
\cite[Prop. 2.4]{P}
respectively.

\end{proof}


Recall the definition \cite[Defn. 3.7]{P}.
If
$U$ as in Definition
\ref{DefnNiceTriple}
then {\it for any $U$-scheme $V$ and any closed point} $u \in U$
set
$V_u=u\times_U V$.
For a finite set $A$ denote $\sharp A$ the cardinality of $A$.

\begin{defn}
\label{Conditions_1*and2*}
Let $(\mathcal X,f,\Delta)$ be a special nice triple
over $U$.
We say that the triple
$(\mathcal X,f,\Delta)$
satisfies conditions $1^*$ and $2^*$
if either the field $k$ is infinite or (if $k$ is finite) the following holds
\begin{itemize}
\item[$(1^*)$]
for $\mathcal Z=\{f=0\} \subset \mathcal X$ and
for any closed point $u \in U$, any integer $d \geq 1$
one has
$$\sharp\{z \in \mathcal Z_u| \text{deg}[k(z):k(u)]=d \} \leq \ \sharp\{x \in \Aff^1_u| \text{deg}[k(x):k(u)]=d \}$$
\item[$(2^*)$]
for the vanishing locus $\mathcal Z$ of $f$
and for any closed point $u \in U$ the point $\Delta(u) \in \mathcal Z_u$ is the only
$k(u)$-rational point of $\mathcal Z_u=u\times_U \mathcal Z$.
\end{itemize}

\end{defn}

\section{Equating group schemes}
\label{SecEquatingGroups}
The main result of the present section is
Theorem \ref{PropEquatingGroups2}.
It is stated and proved at the very end the present section.
We begin with the following result which extends
\cite[Prop.5.1]{PSV}.

\begin{thm}
\label{PropEquatingGroups2}
Let $S$ be a regular semi-local
irreducible scheme.
Assume
that $G_1$ and $G_2$ are reductive $S$-group schemes
which are forms of each other.
Let $T \subset S$ be a connected non-empty closed
sub-scheme of $S$, and
$\varphi: G_1|_T \to G_2|_T$
be $T$-group scheme isomorphism.
Then there exists a finite \'{e}tale morphism
$\tilde S \xra{\pi} S$ together
with
a section $\delta: T \to \tilde S$ of $\pi$ over $T$
and $\tilde S$-group scheme isomorphisms
$\Phi: \pi^*{G_1} \to \pi^*{G_2}$
such that
\begin{itemize}
\item[(i)]
$\delta^*(\Phi)=\varphi$,
\item[(ii)]
the scheme $\tilde S$ is irreducible.
\end{itemize}
\end{thm}

\begin{prop}
\label{PropEquatingGroups}
Theorem
\ref{PropEquatingGroups2}
holds in the case when the group schemes $G_1$ and $G_2$ are semi-simple.
\end{prop}

\begin{proof}[Proof of Proposition \ref{PropEquatingGroups}]
{\bf Let $s\in S$ be a closed point and let $V$ be an $S$-scheme. In this
section we will write $V(s)$ for the scheme $s\times_S V$.
}
The proof of the proposition literally repeats the proof of
\cite[Prop.5.1]{PSV}
except exactly one reference,
which is the reference to
\cite[Lemma 7.2]{OP2}.
That reference one has to replace with the reference to the following
\begin{lem}
\label{OjPan}
Let $S=\text{Spec}(R)$ be a regular
semi-local scheme
{\bf such that
the residue field at any of its closed point is finite.
}
Let $T$ be a closed
subscheme of $S$. Let $\bar W$ be a closed subscheme of
$\Bbb P^d_S=\text{Proj}(S[X_0,\dots,X_d])$ and
$W=\bar W\cap \Aff^d_S$, where $\Bbb A^d_S$ is the affine space defined by
$X_0\neq0$. Let
$W_{\infty}=\bar W\setminus W$ be the intersection of $\bar W$ with the
hyperplane at infinity
$X_0=0$. Assume that over $T$ there exists a section
$\delta:T\to W$
of the canonical projection $W\to S$. Assume further that \\
\smallskip
{\rm{(1)}} $W$ is smooth and equidimensional over $S$,
of relative dimension $r$;\\
{\rm{(2)}} For every closed
point $s\in S$ the closed fibres of $W_\infty$ and $W$
satisfy
$$\text{dim}(W_{\infty}(s))<\text{dim}(W(s))=r\;.$$
\smallskip
Then there exists a closed subscheme $\tilde S$ of $W$ which is finite
\'etale over
$S$ and contains $\delta(T)$.
\end{lem}

\begin{proof}[Proof of Lemma \ref{OjPan}]
To avoid technicalities we will give
the proof in the case of $r=1$ and left the general case to the reader.
If $r=1$, then for every closed
point $s\in S$ the closed fibres of $W_\infty$ and $W$
satisfy
$$\text{dim}(W_{\infty}(s))<\text{dim}(W(s))=1\;.$$

Since $S$ is semi-local, after a linear change
of coordinates we may assume that $\delta$ maps $T$ into
the closed subscheme of $\Pro^d_T$ defined by
$X_1=\dots=X_d=0$. For each closed fibre $\Pro^d(s)$ of $\Pro^d_S$
using
\cite[Thm.1.2]{Poo1},
we can choose a {\bf homogeneous} polynomial $H(s)$
such that the subscheme $Y(s)$ of $\Pro^d(s)$ defined by the equation
$$H(s)=0$$
intersects $W(s)$ transversally
and avoids $W_{\infty}(s)$.

{\bf Let $s\in S$ be a closed point of $S$. Let $s\in T$ be a closed point of $T$.
Let $\Pro^d(s)$ be closed fibre  of $\Pro^d_S$ over the point $s$.
Let $\mathbb A^d(s)\subset \Pro^d(s)$ be the affine subspace defined by the unequality
$X_0\neq 0$.
Let $t_i:=X_i/X_0$ ($i \in \{1,2, \dots, d\}$) be the coordinate function on $\mathbb A^d(s)$.
The origin $x_{0,s}=(0,0,\dots, 0)\in \mathbb A^d(s)$ has the homogeneous coordinates
$[1:0:\dots :0]$ in $\Pro^d(s)$.
Let $W(s)\subset \mathbb P^d(s)$ be the fibre of $W$ over the point $s$.
If $x_{0,s}$ is in $W(s)$, then
let $\tau (s)\subset \mathbb A^d(s)$ be the tangent space to $W(s)$ at the point
$x_0\in \mathbb A^d(s)$.
In this case let $l_s=l_s(t_1,t_2,\dots, t_d)$ be a linear form in $k(s)[t_1,t_2,\cdots,t_d]$
such that
$l_s|_{\tau (s)}\neq 0$.
If $x_{0,s}$ is not in $W(s)$, then set $l_s=t_1$.
In all the cases
$$L_s:=X_0\cdot l_s \in k(s)[X_1,X_2,\cdots,X_d]$$
is a homogeneous polinomial of degree $1$.

Let $s \in S$ be a closed point. Suppose $x_{0,s} \in W(s)$.
Then
by \cite[Thm. 1.2]{Poo1} there is an integer $N_1(s)\geq 1$ such that for any
positive integer $N\geq N_1(s)$ there is a homogeneous polinomial
$$H_{1,N}(s)=X^N_0\cdot F_{0,N}(s)+X^{N-1}_0\cdot F_{1,N}(s)+\dots+X_0\cdot F_{N-1,N}(s)+F_{N,N}(s)\in k(s)[X_0,X_1,\cdots,X_d]$$
of degree $N$ with homogeneous polinomials $F_{i,N}(s)\in k(s)[X_1,\dots,X_n]$ of degree $i$
such that the following holds \\
(i) $F_{0,N}(s)=0$, $F_{1,N}(s)=L_s$; \\
(ii) the subscheme $V(s)\subseteq \Pro^d(s)$ defined by the equation $H_{1,N}(s)=0$
intersects $W(s)$ transversally;\\
(iii) $V(s)\cap W_{\infty}(s)=\emptyset$.

Let $s \in S$ be a closed point. Suppose $x_{0,s}$ is not in $W(s)$.
Then
by \cite[Thm. 1.2]{Poo1} there is an integer $N_1(s)\geq 1$ such that for any
positive integer $N\geq N_1(s)$ there is a homogeneous polinomial
$$H_{1,N}(s)=X^N_0\cdot F_{0,N}(s)+X^{N-1}_0\cdot F_{1,N}(s)+\dots+X_0\cdot F_{N-1,N}(s)+F_{N,N}(s)\in k(s)[X_0,X_1,\cdots,X_d]$$
of degree $N$ with homogeneous polinomials $F_{i,N}(s)\in k(s)[X_1,\dots,X_n]$ of degree $i$
such that the following holds \\
(i) $F_{0,N}(s)=0$, $F_{1,N}(s)=L_s=X_1$; \\
(ii) the subscheme $V(s)\subseteq \Pro^d(s)$ defined by the equation $H_{1,N}(s)=0$
intersects $W(s)$ transversally;\\
(iii) $V(s)\cap W_{\infty}(s)=\emptyset$.

Let $N_1=\text{max}_s \{N_1(s)\}$, where $s$ runs over all the closed points of $S$.
For any closed point $s \in S$ set $H_1(s):= H_{1,N_1}(s)$.
Then for any closed point $s \in S$ the polinomial $H_1(s)\in k(s)[X_0,X_1,\cdots,X_d]$ is homogeneous of the degree $N_1$.
By the chinese
remainders' theorem for any $i=0,1,\dots,N$
there exists a common  lift
$F_{i,N_1}\in A[X_1,\dots,X_d]$ of all polynomials $F_{i,N_1}(s)$,
$s$ is a closed point of $S$,
such that
$F_{0,N_1}=0$ and for any $i=0,1,\dots,N$ the polinomial $F_{i,N_1}$ is homogeneous of
degree $i$.
By the chinese
remainders' theorem
there exists a common lift
$L \in A[X_1,\cdots,X_d]$
of all polynomials $L_s$,
$s$ is a closed point of $S$,
such that
$L$ is homogeneous of the degree one.
Set
$$H_{1,N_1}:=X^{N_1-1}_0\cdot L+X^{N_1-2}_0\cdot F_{2,N_1} \dots+X_0\cdot F_{N_1-1,N_1}+F_{N_1,N_1}\in A[X_0,X_1,\cdots,X_d].$$
Note that for any closed point $s$ in $S$ the evaluation of $H_{1,N_1}$ at the point $s$ coincides with the polinomial
$H_{1,N_1}(s)$ in $k(s)[X_0,X_1,\cdots,X_d]$.
Note also that
$H_{1,N_1}[1:0:\dots:0]=0$.
Hence $H_{1,N_1}|_{\delta(T)}\equiv 0$.

Note that for any closed point $s$ in $S$ the evaluation of $H_{1,N_1}$ at the point $s$ coincides with the polinomial
$H_{1,N_1}(s)$ in $k(s)[X_0,X_1,\cdots,X_d]$.
Since $S$ is semi-local and $W_{\infty}$ is projective quasi-finite over $S$
it is finite over $S$. Let $V\subset \mathbb P^d_S$ be the subscheme defined by
$\{H_{1,N_1}=0\}$.
Since for any closed point $s \in S$ one has
$V(s)\cap W_{\infty}(s)=\emptyset$, hence
$V\cap W_{\infty}=\emptyset$.
Note also that
$H_{1,N_1}[1:0:\dots:0]=0$.
Hence $H_{1,N_1}|_{\delta(T)}\equiv 0$.
}


{\it We claim that the subscheme  $\tilde S=V\cap W$  has the
required properties
}.
Note first that $V\cap W$ is finite
over $S$.
In fact, $V\cap W=V\cap \bar W$, which
is projective over $S$ and such that every closed
fibre (hence every fibre) is finite. Since the closed
fibres of $V\cap W$ are finite \'etale over the closed points
of $S$, to show that $V\cap W$ is finite \'etale over $S$ it
only remains to show that it is flat over $S$. Noting that
$V\cap W \subset W$ is defined in every closed fibre by a length one regular
sequence of equations and localizing at each closed point of
$S$, we see that flatness follows from
\cite[Lemma 7.3]{OP2}. Whence the lemma.
\end{proof}
Return to the proof of the proposition \ref{PropEquatingGroups}.
Its proof literally repeats now the proof of
\cite[Prop.5.1]{PSV}. Apriory the regular scheme $\tilde S$
is not necessary irreducible. In that case replace $\tilde S$
with its irreducible component containing the scheme $\delta(T)$.
The proof of the proposition
\ref{PropEquatingGroups}
is completed.
\end{proof}


\begin{prop}
\label{PropEquatingGroups4}
Theorem \ref{PropEquatingGroups2} holds in the case when the groups $G_1$ and $G_2$ are tori
and, more generally, in the case when the groups $G_1$ and $G_2$ are
of multiplicative type.
\end{prop}
We left a proof of this latter proposition to the reader.
The following proposition follows easily from
\cite[Exp. 9, Cor. 2.9.]{SGA3}.
\begin{prop}
\label{PropEquatingGroups5}
Let $T$ and $S$ be the same as in Theorem \ref{PropEquatingGroups2}.
Let $M_1$ and $M_2$ be two $S$-group schemes of multiplicative type.
Let $\alpha_1, \alpha_2: M_1 \rightrightarrows M_2$
be two $S$-group scheme morphisms such that
$\alpha_1|_T = \alpha_2|_T$.
Then $\alpha_1 = \alpha_2$.
\end{prop}

\begin{proof}[Proof of Theorem \ref{PropEquatingGroups2}]
Let $Rad(G_r) \subset G_r$ be the radical of $G_r$ and let
$der(G_r) \subset G_r$ be the derived subgroup of $G_r$ ($r=1,2$)
(see \cite[Exp.XXII, 4.3]{D-G}).
By the very definition the radical is a torus. The $S$-group scheme
$der(G_r)$ is semi-simple ($r=1,2$).
Set $Z_r:= Rad(G_r) \cap der (G_r)$.
The above embeddings induce natural $S$-group morphisms
$$\Pi_r: Rad(G_r) \times_S der(G_r) \to G_r$$
with $Z_r$ as the kernel ($r=1,2$).
By \cite[Exp.XXII,Prop.6.2.4]{D-G} $\Pi_r$ is a central isogeny.
Particularly, $\Pi_r$ is a faithfully flat finite morphism by
\cite[Exp.XXII,Defn.4.2.9]{D-G}.
Let
$i_r: Z_r \hookrightarrow Rad(G_r) \times_S der(G_r)$
be the closed embedding.

The $T$-group scheme isomorphism
$\varphi: G_1|_T \to G_2|T$
induces certain $T$-group scheme isomorphisms
$\varphi_{der}: der(G_1|_T) \to der(G_2|_T)$,
$\varphi_{rad}: rad(G_1|_T) \to rad(G_2|_T)$
and
$\varphi_Z: Z_1|_T \to Z_2|_T$
{\bf such that}
$$(\Pi_2)|_T \circ (\varphi_{der}\times \varphi_{rad})=\varphi \circ (\Pi_1)|_T \ \text{and} \
i_{2,T} \circ \varphi_Z = (\varphi_{rad} \times \varphi_{der}) \circ i_{1,T}.$$

By Propositions
\ref{PropEquatingGroups}
and
\ref{PropEquatingGroups4}
there exist a finite \'{e}tale morphism
$\pi: \tilde S \to S$
(with an irreducible scheme $\tilde S$)
and its section
$\delta: T \to \tilde S$
over $T$ and $\tilde S$-group scheme isomorphisms
$$\Phi_{der}: der(G_{1, \tilde S}) \to der(G_{2,\tilde S}) \ \ \text{,} \ \
\Phi_{Rad}: Rad(G_{1, \tilde S}) \to Rad(G_{2,\tilde S}) \ \ \text{and} \ \
\Phi_Z: Z_{1, \tilde S} \to Z_{2,\tilde S}$$
such that
$\delta^*(\Phi_{der})= \varphi_{der}$,
$\delta^*(\Phi_{rad})= \varphi_{rad}$
and
$\delta^*(\Phi_Z)= \varphi_Z$.

Since $Z_r$ is contained in the center of $der(G_r)$ and is of multiplicative type
Proposition
\ref{PropEquatingGroups5}
yields the equality
$$i_{2, \tilde S} \circ \Phi_Z = (\Phi_{Rad} \times \Phi_{der}) \circ i_{1, \tilde S}:
Z_{1, \tilde S} \to Rad(G_{2, \tilde S}) \times_{\tilde S} der(G_{2, \tilde S}).$$
Thus $(\Phi_{Rad} \times \Phi_{der})$ induces an $\tilde S$-group scheme isomorphism
$$\Phi: G_{1, \tilde S} \to G_{2, \tilde S}$$
such that
$\Pi_{2, \tilde S} \circ (\Phi_{Rad} \times \Phi_{der})= \Phi \circ \Pi_{1, \tilde S}$.
The latter equality yields the following one
$(\Pi_2)|_T \circ \delta^*(\Phi_{Rad} \times \Phi_{der})= \delta^*(\Phi) \circ (\Pi_1)|_T$,
which in turn yields the equality
$$(\Pi_2)|_T \circ (\varphi_{rad} \times \varphi_{der})= \delta^*(\Phi) \circ (\Pi_1)|_T.$$
Comparing it with the equality
$(\Pi_2)|_T \circ (\varphi_{rad} \times \varphi_{der})= \varphi \circ (\Pi_1)|_T$
and using the fact that
$(\Pi_1)|_T$
is faithfully flat we conclude the equality
$\delta^*(\Phi)= \varphi$.

\end{proof}

\section{Nice triples and group schemes}
\label{NiceTriplesAndGr_Schemes}
We need in an extension of the \cite[Thm. 3.9]{P}.
For that it is convenient to give a definition
under the following set up.
Let $U$ be as in Definition
\ref{DefnNiceTriple}.
Let $(\mathcal X,f,\Delta)$ be a special nice triple
over $U$ and let
$G_{\mathcal X}$ be a reductive
$\mathcal X$-group scheme and
$G_U:= \Delta^*(G_{\mathcal X})$
and
$G_{\text{const}}:=q^*_U(G_U)$.
Let
$\theta: (q^{\prime}: \mathcal X^{\prime} \to U, f^{\prime}, \Delta^{\prime}) \to (q: \mathcal X \to U, f, \Delta)$
be a morphism between nice triples over $U$.

\begin{defn}[Equating group schemes]{\rm
\label{def:q_constant}
We say that
the morphism $\theta$ {\it equates}
the reductive
$\mathcal X$-group schemes
$G_{\mathcal X}$ and $G_{\text{const}}$, if
there is an
$\mathcal X^{\prime}$-group scheme isomorphism
$\Phi: \theta^*(G_{\text{const}}) \to \theta^*(G_{\mathcal X})$
with
$(\Delta^{\prime})^*(\Phi)= id_{G_U}$.
}
\end{defn}

\begin{rem}\label{rem:theta_circ_theta'}
Let $\rho: (\mathcal X'',f'',\Delta'') \to (\mathcal X',f',\Delta')$ and $\theta: (\mathcal X',f',\Delta') \to (\mathcal X,f,\Delta)$
be morphisms of nice triples over $U$.
If $\theta$ equates $G_{\mathcal X}$ and $G_{\text{const}}$,
then $\theta \circ \rho$ also equates
$G_{\mathcal X}$ and $G_{\text{const}}$.
\end{rem}

\begin{thm}
\label{equating3_triples}
Let $U$ be as in Definition
\ref{DefnNiceTriple}. Let $(\mathcal X,f,\Delta)$ be a special nice triple
over $U$.
Let $G_{\mathcal X}$ be a reductive
$\mathcal X$-group scheme and
$G_U:= \Delta^*(G_{\mathcal X})$
and
$G_{\text{const}}:=q^*_U(G_U)$.
Then
there exist a morphism
$\theta'': (q^{\prime\prime}: \mathcal X^{\prime\prime} \to U, f^{\prime\prime}, \Delta^{\prime\prime}) \to
(q: \mathcal X \to U, f, \Delta)$
between nice triples over $U$ such that
\begin{itemize}
\item[(i)]
the morphism $\theta''$ equates the reductive $\mathcal X$-group schemes
$G_{\const}$ and $G_{\mathcal X}$
\item[(ii)]
the triple
$(\mathcal X^{\prime\prime},f^{\prime\prime},\Delta^{\prime\prime})$
is a special nice triple
over $U$ subjecting to the conditions
$(1^*)$ and $(2^*)$ from Definition
\ref{Conditions_1*and2*}.
\end{itemize}
\end{thm}

\begin{proof}[Proof of Theorem \ref{equating3_triples}]
Let $U$ be as in the theorem. Let
$(\mathcal X,f,\Delta)$
be a special nice triple
over $U$ as in the theorem.
By the definition
of a nice triple there exists a finite surjective morphism
$\Pi:\mathcal X\to\Aff^1\times U$ of $U$-schemes.
The first part of the construction \cite[Constr. 4.2]{P}
gives us now the data
$(\mathcal Z,\mathcal Y, S, T)$,
where
$(\mathcal Z,\mathcal Y, T)$ are closed subsets of $\mathcal X$ finite over $U$.
If $\{y_1,\dots,y_n\}$ are all the closed points of $\mathcal Y$,
then $S=\text{Spec}(\mathcal O_{X, y_1,...,y_n})$.

Further, let $G_U=\Delta^*(G_{\mathcal
X})$ be as in the hypotheses of Theorem
\ref{equating3_triples} and
let
$G_{\const}$
be the pull-back of
$G_U$ to $\mathcal X$. Finally,
let
$\varphi:G_{\const}|_T \to G_{\mathcal X}|_T$
be the canonical
isomorphism. Recall that by assumption $\mathcal X$ is $U$-smooth and irreducible,
and thus $S$ is regular and irreducible.
By
Theorems
\ref{PropEquatingGroups2}
there exists a finite
\'etale morphism $\theta_0:S^{\prime}\to S$, a section
$\delta:T\to S^{\prime}$ of $\theta_0$ over $T$ and an isomorphism
$\Phi_0:\theta^*_0(G_{\const,S})\to\theta^*_0(G_{\mathcal X}|_S)$
such that
$\delta^*(\Phi_0)=\varphi$,
and
{\it where the scheme $S^{\prime}$ is irreducible.
}

Consider now the diagram (4) from the construction \cite[Constr.4.2]{P}.
\begin{equation}
\label{Diag_Manuel}
\xymatrix{
     {}  &  S^{\prime} \ar[d]^{\theta_0} \ar@{^{(}->}@<-2pt>[r]  & {\mathcal V}^{\prime}  \ar[d]_{\theta} &\\
     T \ar@{^{(}->}@<-2pt>[r] \ar[ur]^{
\delta } & S \ar@{^{(}->}@<-2pt>[r]  &   \mathcal V \ar@{^{(}->}@<-2pt>[r]  &  \mathcal X &\\
    }
\end{equation}
\noindent
Recall that here
$\theta: {\mathcal V}^{\prime}\to\mathcal V$
is finite \'etale (and the square is cartesian).
We may and will now suppose that
the neighborhood $\mathcal V$ of the points
$\{y_1,\dots,y_n\}$
from that diagram is chosen such that
there is $\mathcal V'$-group schemes isomorphism
$\Phi: \theta^*(G_{\const,\mathcal V})\to\theta^*(G_{\mathcal X}|_{\mathcal V})$
with $\Phi|_{S'}=\Phi_0$.
Clearly,
$\delta^*(\Phi)=\varphi$.

Applying the second part of the construction \cite[Constr.4.2]{P}
and also the proposition
\cite[Prop. 4.3]{P}
to the finite
\'etale morphism $\theta: \mathcal V^{\prime}\to \mathcal V$ and to the section
$\delta:T\to \mathcal V^{\prime}$ of $\theta$ over $T$
we get \\
0)  firstly, an open subset $\mathcal W\subseteq\mathcal V$ containing $\mathcal Y$ (and hence containing $S$)
and endowed with a
finite surjective $U$-morphism $\Pi^*:\mathcal W\to\Aff^1\times
U$;\\
1) secondly, a triple $(\mathcal X^{\prime},f^{\prime},\Delta^{\prime})$;\\
2) thirdly, the \'{e}tale morphism of $U$-schemes $\theta: \mathcal X^{\prime} \to \mathcal X$;\\
3) forthly, inclusions of $U$-schemes $S\subset \mathcal W$ and $S' \subset \mathcal X'$.\\
Further we get \\
(i) the special nice triple $(q_U\circ \theta: \mathcal X^{\prime}\to U,f^{\prime},\Delta^{\prime})$
over $U$;\\
(ii) the morphism $\theta$ is a morphism
$(\mathcal X^{\prime},f^{\prime},\Delta^{\prime}) \to (\mathcal X,f,\Delta)$
between the nice triples, which {\it equates} the $\mathcal X$-group schemes
$G_{\const}$
and
$G_{\mathcal X}$;\\
(iii) the equality $f^{\prime}=\theta^*(f)$.

To complete the proof of the theorem just apply
the theorem \cite[Thm. 3.9]{P} to the
the special nice triple
$(\mathcal X^{\prime},f^{\prime},\Delta^{\prime})$
and use the remark \ref{rem:theta_circ_theta'}.
\end{proof}

\section{First application of the theory of nice triples}
\label{First Application}
\begin{thm}\label{equating3}
Let $X$ be an affine $k$-smooth irreducible $k$-variety, and let $x_1,x_2,\dots,x_n$ be closed points in $X$.
Let $U=Spec(\mathcal O_{X,\{x_1,x_2,\dots,x_n\}})$.
Let $\bG$ be a reductive
$X$-group scheme
and let
$\bG_U= can^*(\bG)$
be the pull-back of $\bG$ to $U$.
Given a non-zero function $\textrm{f}\in k[X]$ vanishing at each point $x_i$,
there is a diagram of the form
\begin{equation}
\label{DeformationDiagram0}
    \xymatrix{
\Aff^1 \times U\ar[drr]_{\pr_U}&&\mathcal X \ar[d]^{}
\ar[ll]_{\sigma}\ar[d]_{q_U}
\ar[rr]^{q_X}&&X &\\
&&U \ar[urr]_{\can}\ar@/_0.8pc/[u]_{\Delta} &\\
    }
\end{equation}
with an irreducible {\bf affine} scheme $\mathcal X$, a smooth morphism $q_U$, a finite surjective $U$-morphism $\sigma$ and an essentially smooth morphism $q_X$,
and a function $f^{\prime} \in q^*_X(\textrm{f} \ )k[\mathcal X]$,
which enjoys the following properties:
\begin{itemize}
\item[\rm{(a)}]
if
$\mathcal Z^{\prime}$ is the closed subscheme of $\mathcal X$ defined by the principal ideal
$(f^{\prime})$, the morphism
$\sigma|_{\mathcal Z^{\prime}}: \mathcal Z^{\prime} \to \Aff^1\times U$
is a closed embedding and the morphism
$q_U|_{\mathcal Z^{\prime}}: \mathcal Z^{\prime} \to U$ is finite;
\item[\rm{(a')}] $q_U\circ \Delta=id_U$ and $q_X\circ \Delta=can$ and $\sigma\circ \Delta=i_0$ \\
(the first equality shows that $\Delta(U)$ is a closed subscheme in $\mathcal X$);
\item[\rm{(b)}] $\sigma$
is \'{e}tale in a neighborhood of
$\mathcal Z^{\prime}\cup \Delta(U)$;
\item[\rm{(c)}]
$\sigma^{-1}(\sigma(\mathcal Z^{\prime}))=\mathcal Z^{\prime}\coprod \mathcal Z^{\prime\prime}$
scheme theoretically
for some closed subscheme $\mathcal Z^{\prime\prime}$
\\ and
$\mathcal Z^{\prime\prime} \cap \Delta(U)=\emptyset$;
\item[\rm{(d)}]
$\mathcal D_0:=\sigma^{-1}(\{0\} \times U)=\Delta(U)\coprod \mathcal D^{\prime}_0$
scheme theoretically
for some closed subscheme $\mathcal D^{\prime}_0$
and $\mathcal D^{\prime}_0 \cap \mathcal Z^{\prime}=\emptyset$;
\item[\rm{(e)}]
for $\mathcal D_1:=\sigma^{-1}(\{1\} \times U)$ one has
$\mathcal D_1 \cap \mathcal Z^{\prime}=\emptyset$.
\item[\rm{(f)}]
there is a monic polinomial
$h \in \mathcal O[t]$
such that
$(h)=Ker[\mathcal O[t] \xrightarrow{\sigma^*} k[\mathcal X] \xrightarrow{-} k[\mathcal X]/(f^{\prime})]$, \\
where $\mathcal O:=k[U]$ and the map bar takes any $g\in k[\mathcal X]$ to ${\bar g}\in k[\mathcal X]/(f^{\prime})$;\\
\item[\rm{(g)}] there is an $\mathcal X$-group scheme isomorphism
$\Phi:  p^*_U(\bG_U)\to p^*_X(\bG)$
with
$\Delta^*(\Phi)= id_{\bG_U}$.
\end{itemize}
\end{thm}


\begin{proof}[Proof of Theorem \ref{equating3}]
By Proposition \ref{BasicTripleProp_1} one can shrink $X$ such that
$x_1,x_2, \dots , x_n$ are still in $X$ and $X$ is affine, and then to construct a special nice triple
$(q_U: \mathcal X \to U, \Delta, f)$ over $U$ and an essentially smooth morphism $q_X: \mathcal X \to X$ such that
$q_X \circ \Delta= can$, $f=q^*_X(\text{f})$ and the set of closed points of $\Delta(U)$ is
contained in the set of closed points of $\{f=0\}$.

Set $\bG_{\mathcal X}=q^*_X(\bG)$, then $\Delta^*(\bG_{\mathcal X})=can^*(\bG)$. Thus the $U$-group scheme
$\bG_U$ from Theorem \ref{equating3_triples} and the $U$-group scheme
$\bG_U$ from Theorem \ref{equating3} are the same. By Theorem
\ref{equating3_triples}
there exists a morphism
$\theta:(\mathcal X_{new},f_{new},\Delta_{new})\to(\mathcal X,f,\Delta)$
such that the triple
$(\mathcal X_{new},f_{new},\Delta_{new})$
is a special nice triple
over $U$
subject to the conditions
$(1^*)$ and $(2^*)$ from Definition
\ref{Conditions_1*and2*}.
And, additionally,
there is an isomorphism
$$\Phi: (q_U\circ \theta)^*(\bG_U)=\theta^*(G_{\text{const}})
\to \theta^*(G_{\mathcal X})=(q_X \circ \theta)^*(\bG) \ \text{with} \ (\Delta_{new})^*(\Phi)= id_{G_U}$$

The triple
$(\mathcal X_{new},f_{new},\Delta_{new})$
is a special nice triple
{\bf over} $U$
subject to the conditions
$(1^*)$ and $(2^*)$ from Definition
\ref{Conditions_1*and2*}.
Thus by \cite[Thm. 3.8]{P}
there is a finite surjective morphism
$\Aff^1\times U \xleftarrow{\sigma_{new}} \mathcal X_{new}$
of the $U$-schemes satisfying the conditions
$(a)$ to $(\textrm{f})$
from Theorem \ref{equating3}. Hence one has a diagram of the form
\begin{equation}
\label{DeformationDiagram0_1}
    \xymatrix{
\Aff^1 \times U\ar[drr]_{\pr_U}&&\mathcal X_{new} \ar[d]^{}
\ar[ll]_{\sigma_{new}}\ar[d]_{q_U\circ \theta}
\ar[rr]^{q_X\circ \theta}&&X &\\
&&U \ar[urr]_{\can}\ar@/_0.8pc/[u]_{\Delta^{\prime}} &\\
    }
\end{equation}
with the irreducible scheme $\mathcal X_{new}$, the smooth morphism $q_{U,new}:=q_U\circ \theta$,
the finite surjective morphism $\sigma_{new}$ and the essentially smooth morphism $q_{X,new}:=q_X\circ \theta$
and with the function
$f_{new} \in (q_{X,new})^*(\textrm{f})k[\mathcal X_{new}]$,
which after identifying notation enjoy the properties
(a) to (\textrm{f}) from Theorem \ref{equating3}.
The isomorphism $\Phi$ is
the desired ones.
Whence the Theorem \ref{equating3}.
\end{proof}

We keep notation of the theorems \ref{equating3}.
To formulate a consequence of the theorem \ref{equating3} (see Corollary \ref{ElementaryNisSquareNew_12} below),
note that using the items (b) and (c) of Theorem
\ref{equating3}
one can find an element
$g \in I(\mathcal Z^{\prime\prime})$
such that \\
(1) $(f^{\prime})+(g)=\Gamma(\mathcal X, \mathcal O_{\mathcal X})$, \\
(2) $Ker(\Delta^*)+(g)=\Gamma(\mathcal X_{new}, \mathcal O_{\mathcal X_{new}})$, \\
(3) $\sigma_{g}=\sigma|_{\mathcal X_g}: \mathcal X_g \to \Aff^1_U$ is \'{e}tale.\\

\begin{cor}[Corollary of Theorem \ref{equating3}]
\label{ElementaryNisSquareNew_12}
The function $f^{\prime}$ from Theorem \ref{equating3}, the polinomial $h$ from the item $(\textrm{f} \ )$
of that Theorem, the morphism $\sigma: \mathcal X \to \Aff^1_U$
and the function
$g \in \Gamma(\mathcal X,\mathcal O_{\mathcal X} )$
defined just above
enjoy the following properties:
\begin{itemize}
\item[\rm{(i)}]
the morphism
$\sigma_g= \sigma|_{\mathcal X_g}: \mathcal X_g \to \Aff^1\times U $
is \'{e}tale,
\item[\rm{(ii)}]
data
$ (\mathcal O[t],\sigma^*_g: \mathcal O[t] \to \Gamma(\mathcal X,\mathcal O_{\mathcal X})_g, h ) $
satisfies the hypotheses of
\cite[Prop.2.6]{C-TO},
i.e.
$\Gamma(\mathcal X,\mathcal O_{\mathcal X} )_g$
is a finitely generated
$\mathcal O[t]$-algebra, the element $(\sigma_g)^*(h)$
is not a zero-divisor in
$\Gamma(\mathcal X,\mathcal O_{\mathcal X} )_g$
and
$\mathcal O[t]/(h)=\Gamma(\mathcal X,\mathcal O_{\mathcal X})_g/h\Gamma(\mathcal X,\mathcal O_{\mathcal X})_g$ \ ,
\item[\rm{(iii)}]
$(\Delta(U) \cup \mathcal Z') \subset \mathcal X_g$ \ and $\sigma_g \circ \Delta=i_0: U\to \Aff^1\times U$,
\item[\rm{(iv)}]
$\mathcal X_{gh} \subseteq \mathcal X_{gf^{\prime}}\subseteq \mathcal X_{f^{\prime}}\subseteq \mathcal X_{q^*_X(\textrm{f})}$ \ ,
\item[\rm{(v)}]
$\mathcal O[t]/(h)=\Gamma(\mathcal X,\mathcal O_{\mathcal X})/(f^{\prime})$
and
$h\Gamma(\mathcal X,\mathcal O_{\mathcal X})=(f^{\prime})\cap I(\mathcal Z^{\prime\prime})$
and
$(f^{\prime}) +I(\mathcal Z^{\prime\prime})=\Gamma(\mathcal X,\mathcal O_{\mathcal X})$.
\end{itemize}
\end{cor}

\begin{proof}
Just repeat literally the proof of \cite[Cor. 7.2]{P}.
\end{proof}

\begin{rem}
\label{ElementaryNisSquareRem_12}
The item \rm{(ii)} of this corollary shows that the cartesian square
\begin{equation}
\label{SquareDiagram2_1}
    \xymatrix{
     \mathcal X_{gh}  \ar[rr]^{\inc} \ar[d]_{\sigma_{gh}} &&  \mathcal X_g \ar[d]^{\sigma_g}  &\\
     (\Aff^1 \times U)_{h} \ar[rr]^{\inc} && \Aff^1 \times U &\\
    }
\end{equation}
can be used to glue principal $\bG$-bundles for a reductive $U$-group scheme $\bG$.
\end{rem}

Set $Y:=\mathcal X_g$, $p_X=q_X: Y\to X$, $p_U=q_U: Y\to U$, $\tau=\sigma_g$, $\tau_h=\sigma_{gh}$, $\delta=\Delta:U \to Y$ and note that
$pr_U\circ \tau=p_U$. Take the monic polinomial
$h \in \mathcal O[t]$ from the item (f) of Theorem \ref{equating3}.
With this replacement of notation and with the element $h$ we arrive to the following

\begin{thm}\label{Major2}
Let the field $k$, the variety $X$, its closed points $x_1,x_2,\dots,x_n$, the semi-local ring $\mathcal O=\mathcal O_{X,\{x_1,x_2,\dots,x_n\}}$,
the semi-local scheme $U=Spec(\mathcal O)$,
the function $\textrm{f}\in k[X]$
be the same as in Theorem \ref{equating3}.
Let $\bG$ be a reductive
$X$-group scheme
and let
$\bG_U= can^*(\bG)$
be the pull-back of $\bG$ to $U$.
Then one has a well-defined commutative diagram of affine schemes with the irreducible affine $U$-smooth $Y$,
a section $\delta: U\to Y$ of the structure morphism $p_U: Y\to U$,
and the monic polinomial $h\in O[t]$
\begin{equation}
\label{SquareDiagram2_2_2}
    \xymatrix{
       (\Aff^1 \times U)_{h}  \ar[d]_{inc} && Y_h:=Y_{\tau^*(h)} \ar[ll]_{\tau_{h}}  \ar[d]^{inc} \ar[rr]^{(p_X)|_{Y_h}} && X_{\text{f}}  \ar[d]_{inc}   &\\
     (\Aff^1 \times U)  && Y  \ar[ll]_{\tau} \ar[rr]^{p_X} && X                                     &\\
    }
\end{equation}
subject to the following conditions:
\begin{itemize}
\item[\rm{(i)}]
the left hand side square
is an elementary {\bf distinguished} square in the category of affine $U$-smooth schemes in the sense of
\cite[Defn.3.1.3]{MV};
\item[\rm{(ii)}]
$p_X\circ \delta=can: U \to X$, where $can$ is the canonical morphism,
\item[\rm{(iii)}]
$\tau\circ \delta=i_0: U\to \Aff^1 \times U$ is the zero section
of the projection $pr_U: \Aff^1 \times U \to U$;
\item[\rm{(iv)}] $h(1)\in \mathcal O[t]$ is a unit;
\item[\rm{(v)}]
there is a $Y$-group scheme isomorphism
$\Phi:  p^*_U(\bG_U)\to p^*_X(\bG)$
with $\delta^*(\Phi)= id_{\bG_U}$.
\end{itemize}
\end{thm}

\begin{proof}
The items \rm{(i)} and \rm{(iv)} of
the Corollary \ref{ElementaryNisSquareNew_12}
show that the morphisms $\delta(U): U\to Y$ and $(p_X)|_{Y_h}: Y_h\to X_{\text{f}}$ are well defined.
The items \rm{(i)}, \rm{(ii)} of that Corollary show that the left hand side square in the diagram
(\ref{SquareDiagram2_2_2})
is an elementary {\bf distinguished} square in the category of smooth $U$-schemes in the sense of
\cite[Defn.3.1.3]{MV}.
The equalities $p_X\circ \delta=can$ and $\tau\circ \delta=i_0$ are obvious.
The property (iv) of the polinomial $h$ follows from the items
(e),(f) and (a) of Theorem \ref{equating3}.
The property (v) of the isomorphism $\Phi$ follows from
the item (g) of  Theorem \ref{equating3}.
\end{proof}


\section{Second application of the theory of nice triples}
\label{Section_Proof_Of_Theorem_1_5}
\begin{proof}[Proof of Theorem \ref{MainHomotopyIntrod}]
The $k$-algebra $\mathcal O$ is of the form
$\mathcal O_{X, \{x_1,x_2,\dots,x_n\}}$,
where $X$ is a $k$-smooth irreducible affine variety.
We may and will assume in this proof that
the reductive group scheme $\bG$ and the principal $\bG$-bundle
$\mathcal G$
are both defined over the variety $X$.
Futhermore we may and will assume that
there is given a non-zero function
$\textrm{f} \in k[X]$
such that the $\bG$-bundle
${\mathcal G}$ is trivial on
$X_\textrm{f}$
and the function
$\textrm{f}$
vanishes at each point $x_i$ in
$\{x_1,x_2,\dots,x_n\}$.
By Theorem
\ref{Major2}
there is a diagram of the form (\ref{SquareDiagram2_2})
enjoying  the properties (i) to (iv) from Theorem
\ref{Major2}. Moreover there
is a $Y$-group scheme isomorphism
$\Phi:  p^*_U(\bG_U)\to p^*_X(\bG)$
such that
$\delta^*(\Phi)= id_{\bG_U}$.

Consider the commutative diagram (\ref{SquareDiagram2_2_2}).
Given a $\bG$-bundle $\mathcal G$ over $X$, which is trivial
on $X_{\textrm{f}}$ take its pull-back $p^{*}_X(\mathcal G)$ to $Y$.
Using the isomorphism $\Phi$ we may and will regard
the $p^*_X(\bG)$-bundle $p^{*}_X(\mathcal G)$
as a $p^*_U(\bG_U)$-bundle,
t.e. as a $\bG_U$-bundle.
We will denote that
$\bG_U$-bundle
by
$_{_{U}}p^{*}_X(\mathcal G)$.

The $\bG$-bundle is trivial on $X_{\text{f}}$.
Hence
the $p^*_X(\bG)$-bundle $p^{*}_X(\mathcal G)$
is trivial on
$Y_{p^*_X{(\textrm{f}})}$.
Thus
the $\bG_U$-bundle $_{_{U}}p^{*}_X(\mathcal G)$
is trivial on
$Y_{p^*_X{(\textrm{f}})}$.
Hence
it is trivial also on $Y_h$.

Take a trivial $\bG_U$-bundle over $(\Aff^1_{U})_h$ and glue it with the $\bG_U$-bundle
$_{_{U}}q^{*}_X(\mathcal G)|_Y$ patching over $Y_h$
(it can be done due to Theorem
\ref{Major2}(i)
).
We get a
$\bG_U$-bundle $\mathcal G_t$ over $\Aff^1_{U}$ which has particularly the following properties:\\
$(a)$ the restriction of $\mathcal G_t$ to $(\Aff^1_{U})_h$ is trivial (by the construction);\\
$(b)$ there is an isomorphism $\psi:\ _{U}q^{*}_X(\mathcal G)|_Y \to \sigma^*_g(\mathcal G_t)$ of the $\bG_U$-bundles;\\
It remains to check that
the restriction of the $\bG_U$-bundle
$\mathcal G_t$ to $0\times U$ is isomorphic to the
$\bG_U$-bundle $can^*(\mathcal G)$. To do that note that
Theorem
\ref{Major2}(ii) and Theorem
\ref{Major2}(iii)
yield the equalities
$$\bG_U=\delta^*(q^*_X(\bG)) \ \ \text{and} \ \ can^*(\mathcal G)=\delta^*(q^*_X(\mathcal G)).$$
There are two interesting $\bG_U$-bundles over $U$. Namely, the $\bG_U$-bundle
$can^*(\mathcal G)=\delta^*(q^*_X(\mathcal G))$
and the $\bG_U$-bundle $\delta^*(_{_{U}}q^{*}_X(\mathcal G))$.
{\it They coincide} since $\delta^*(\Phi)=id_{\bG_U}$.
Thus
$$can^*(\mathcal G)=\delta^*(_{_{U}}q^{*}_X(\mathcal G))\cong \delta^*(\sigma^*_g(\mathcal G_t))={\mathcal G_t}|_{0\times U},$$
where the middle $\bG_U$-bundle isomorphism is the isomorphism
$\delta^*(\psi)$.
The latter equality holds
by Theorem
\ref{Major2}(iii).
Whence the Theorem \ref{MainHomotopyIntrod}.

\end{proof}



\begin{rem}
Here is {\it the motivic view point} on the above arguments (in the constant case).
The distinguished elementary square (\ref{SquareDiagram2_1})
defines a motivic space isomorphism
$\mathcal X_g/\mathcal X_{gh}\xleftarrow{\sigma} \Aff^1_U/(\Aff^1_U)_{h}$
(just a Nisnevich sheaf isomorphism),
hence there is a composite morphism of motivic spaces of the form
$$\varphi: \Aff^1_U/(\Aff^1_U)_{h} \xrightarrow{\sigma^{-1}}
\mathcal X_g/\mathcal X_{gh} \to \mathcal X_g/\mathcal X_{q^*_X{(\textrm{f}})}  \xrightarrow{q}  X/X_{\textrm{f}}.$$
Let $i_0: 0\times U \to \Aff^1_U/(\Aff^1_U)_h$ be the natural morphism. By the properties (a') and (d) from Theorem
\ref{equating3}
the morphism
$\varphi \circ i_0$ equals to the one
$$U\xrightarrow{can} X \xrightarrow{p} X/X_{\textrm{f}},$$
where
$p: X\to X/X_{\textrm{f}}$
is the canonical morphisms.

Now assume that $\bG_0$ is a reductive group scheme over the field $k$.
A $\bG_0$-bundle over $X$, trivialized on $X_{\textrm{f}}$, is "classified" by a morphism
$\rho: X/X_{\textrm{f}} \to (B\bG_0)_{et}$ in an appropriate category.
Thus the morphism
$\rho\circ \varphi$
"classifies" a $\bG_0$-bundle $\mathcal G_t$ over $\Aff^1_U$ trivialized on $(\Aff^1_U)_h$.
The equality
$\varphi \circ i_0=p\circ can$
shows that the $\bG_0$-bundles $\mathcal G_t|_{0\times U}$ and $can^*(\mathcal G)$
are isomorphic. This "proves" Theorem \ref{MainHomotopyIntrod} in the constant case.
\end{rem}

\section{An extension of Theorem \cite[Thm 1.1]{PSV}}\label{PSVfinite_f}
\begin{thm}\label{PSV}
Let $k$ be a field. Let $\mathcal O$ be the semi-local ring of finitely many closed points on a $k$-smooth irreducible affine $k$-variety $X$
and let $K$ be its field of fractions. Let $\bG$ be an isotropic simple simply connected group scheme over $\mathcal O$.
Then for any Noetherian $k$-algebra $A$ the map
$$\text{H}^1_{et}(\mathcal O \otimes_k  A,G) \to \text{H}^1_{et}(K\otimes_k  A,G),$$
induced by the inclusion $\mathcal O$ into $K$, has trivial kernel.
\end{thm}

\begin{proof}
If the field $k$ is infinite, this theorem is exactly Theorem
\cite[Thm 1.1]{PSV}. So, there is nothing to prove in this case.
If the field $k$ is finite, then
repeat literally the proof of
\cite[Thm 1.1]{PSV} and replace the
reference to
\cite[Thm 1.2]{PSV}
with the reference to Theorem \ref{PSV2}.
\end{proof}

\begin{thm}\label{PSV2}
Let $k$, $\mathcal O$, $K$, $A$ be the same as in Theorem \ref{PSV}.
Let $G$ be a not necessarily isotropic simple simply connected group
scheme over $\mathcal O$.
Let $\mathcal G$ be a principal $G$-bundle over $\mathcal O \otimes_k A$ which is trivial
over $K \otimes_k A$.
Then there exists a principal $G$-bundle $\mathcal G_t$ over
$\mathcal O[t] \otimes_k A$
and a monic polynomial
$f(t) \in \mathcal O[t]$
such that
\par
(i) the $G$-bundle $\mathcal G_t$ is trivial over $(\mathcal O[t]_f) \otimes_k A$,
\par
(ii) the evaluation of $\mathcal G_t$ at $t=0$ coincides
with the original $G$-bundle $\mathcal G$,
\par
(iii) $f(1) \in \mathcal O$ is invertible in $\mathcal O$.
\end{thm}

\begin{proof}[Proof of Theorem \ref{PSV2}]
If $A=k$, then Theorem \ref{PSV2} coincides with Theorem
\ref{MainHomotopyIntrod} and there is nothing to prove.
The general case we left to the reader (follow literally the arguments
from the proof of Theorem
\ref{MainHomotopyIntrod}
given in Section
\ref{Section_Proof_Of_Theorem_1_5}
).

\end{proof}

\end{document}